\documentclass[12pt, reqno]{amsart}
\usepackage{amsmath, amssymb, amsthm, fullpage, mathrsfs}

\newtheorem{thm}{Theorem}[section]
\newtheorem{corollary}[thm]{Corollary}
\newtheorem{lemma}[thm]{Lemma}
\newtheorem{theorem}[thm]{Theorem}

\numberwithin{equation}{section}

\theoremstyle{definition}
\newtheorem{rem}[thm]{Remark}

\newcommand{\al}{\alpha}
\renewcommand{\b}{\beta}
\newcommand{\de}{\delta}

\newcommand{\la}{\lambda}
\renewcommand{\phi}{\varphi}

\renewcommand{\d}{\partial}
\newcommand{\R}{{\mathbb R}}
\newcommand{\Case}[1]{\noindent \underline{Case #1.}}

\renewcommand{\qed}{\rule{3mm}{3mm}}
\renewenvironment{proof}
    {\vspace{1mm}\noindent\textbf{Proof.}}
    {\hspace*{\fill} $\qed$\vspace{1mm}}
\newenvironment{proof_of}[1]
    {\vspace{1mm}\noindent {\bf Proof of #1.}}
    {\hspace*{\fill} $\qed$\vspace{1mm}}

\begin{document}
\title[Asymptotic expansion of radial solutions]
{Asymptotic expansion of radial solutions for supercritical biharmonic equations}
\author{Paschalis Karageorgis}
\address{School of Mathematics, Trinity College, Dublin 2, Ireland.}
\email{pete@maths.tcd.ie}

\begin{abstract}
Consider the positive, radial solutions of the nonlinear biharmonic equation $\Delta^2 \phi = \phi^p$.  There is a critical
power $p_c$ such that solutions are linearly stable if and only if $p\geq p_c$.  We obtain their asymptotic expansion at
infinity in the case that $p\geq p_c$.
\end{abstract}
\maketitle

\section{Introduction}
We study the positive, radial solutions of the nonlinear biharmonic equation
\begin{equation}\label{ee4}
\Delta^2 \phi(x) = \phi(x)^p, \qquad x\in \R^n.
\end{equation}
Such solutions are known to exist when $n>4$ and $p\geq \frac{n+4}{n-4}$, but they fail to exist, otherwise.  Our goal in this
paper is to derive their asymptotic expansion as $|x|\to\infty$ and thus obtain an analogue of a well-known result \cite{GNW}
for the second-order equation
\begin{equation}\label{ee2}
-\Delta \phi (x) = \phi(x)^p, \qquad x\in \R^n.
\end{equation}
We remark that the qualitative properties of solutions to \eqref{ee4} resemble those of solutions to \eqref{ee2}, however the
methods used to establish them are quite different.

First, let us summarize the known results for the second-order equation \eqref{ee2}.  If $n\leq 2$ or $1<p<\frac{n+2}{n-2}$,
then no positive solutions exist; and if $p=\frac{n+2}{n-2}$, then all positive solutions are radial up to a translation and
also explicit \cite{CGS, CL}.  If $p>\frac{n+2}{n-2}$, finally, the positive radial solutions form a one-parameter family
$\{\phi_\al\}_{\al>0}$, see \cite{DN, Li}.  When it comes to the behavior of the solutions $\phi_\al$, a crucial role is
played by a singular solution of the form
\begin{equation}\label{Phi2}
\Phi(x) = a_0(n,p) \cdot |x|^{-\frac{2}{p-1}}.
\end{equation}
As $|x|\to\infty$, that is, each $\phi_\al$ behaves like the singular solution $\Phi$.

There is also a critical value $p_c$ associated with the second-order equation \eqref{ee2}.  This is defined by taking
$(\frac{n+2}{n-2}, p_c)$ to be the maximal interval on which
\begin{equation}\label{pc2}
p \cdot Q_2\left( \frac{2}{p-1} \right) > Q_2\left( \frac{n-2}{2} \right),
\qquad Q_2(\al) \equiv |x|^{\al+2} (-\Delta) \,|x|^{-\al}.
\end{equation}
It is easy to check that $p_c=\infty$ if $n\leq 10$ and that $p_c<\infty$, otherwise.  In the subcritical case
$\frac{n+2}{n-2}< p<p_c$, each radial solution $\phi_\al$ oscillates around the singular solution \eqref{Phi2} and the graphs
of any two radial solutions intersect one another \cite{Wang}.  In the supercritical case $p\geq p_c$, on the other hand, the
graphs of distinct solutions $\phi_\al$ do not intersect one another and they do not intersect the graph of the singular
solution, either \cite{Wang}.

Let us now turn to the fourth-order equation \eqref{ee4}.  Although the known results are very similar to those listed above,
their proofs are generally quite different.  In this case, positive solutions fail to exist if $n\leq 4$ or $1<p<
\frac{n+4}{n-4}$, while they are explicit and radial up to a translation, if $p=\frac{n+4}{n-4}$, see \cite{CLO, Lin, WX}.  And if
$p>\frac{n+4}{n-4}$, there is a one-parameter family of radial solutions $\phi_\al$ which behave asymptotically like a
singular solution of the form
\begin{equation}\label{Phi4}
\Phi(x) = a_0(n,p) \cdot |x|^{-\frac{4}{p-1}},
\end{equation}
see \cite{GG}.  The associated critical value arose in \cite{GG} and it is defined by taking $(\frac{n+4}{n-4}, p_c)$ to be
the maximal interval on which
\begin{equation}\label{pc4}
p \cdot Q_4\left( \frac{4}{p-1} \right) > Q_4\left( \frac{n-4}{2} \right),
\qquad Q_4(\al) \equiv |x|^{\al+4} \Delta^2 |x|^{-\al}.
\end{equation}
Moreover, $p_c<\infty$ if and only if $n\geq 13$, while the graphs of radial solutions intersect one another in the
subcritical case \cite{FGK, Wei2} but not in the supercritical case \cite{Wei2, Kar}.

There are also results that are well-known in the second-order case \eqref{ee2} but remain open in the fourth-order case
\eqref{ee4}.  Those include the precise asymptotic expansion of the radial solutions in the supercritical case
$p\geq p_c$.  Expansions for the second-order case go back to \cite{GNW, Li} and they provided a key ingredient for studying
the stability of steady states for the nonlinear heat equation $u_t - \Delta u = u^p$ in the supercritical case \cite{GNW}.
Our goal in this paper is to establish an analogous expansion for the fourth-order problem.

\begin{theorem}\label{main}
Suppose $n\geq 13$ and $p\geq p_c$.  Let $\phi$ be a positive, radial solution of \eqref{ee4} and let $\la_1 < \la_2 \leq
\la_3 < 0 < \la_4$ be the eigenvalues of the associated linearized equation, see Lemma \ref{fact}.  Then there exists a
finite sequence $p_c = p_1 < p_2 < \cdots < p_N$ such that $\la_2\leq k\la_3$ if and only if $p\geq p_k$.  Moreover, $\phi$
has the following asymptotic expansion as $r\to \infty$.

\begin{itemize}
\item[(a)]
If $p_k< p < p_{k+1}$, with the convention that $p_{N+1}=\infty$, then
\begin{equation}\label{main1}
r^{\frac{4}{p-1}} \phi(r) = a_0 + \sum_{j=1}^k a_j r^{j\la_3} + b_1 r^{\la_2} + a_{k+1}r^{(k+1)\la_3} +  O
\left( r^{\la_2+ \la_3} \right).
\end{equation}

\item[(b)]
If $p=p_k$ for some $k\geq 2$, then $k\la_3=\la_2$ and
\begin{equation}\label{main2}
r^{\frac{4}{p-1}} \phi(r) = a_0 + \sum_{j=1}^{k-1} a_j r^{j\la_3} + r^{k\la_3} (b_1\log r + a_k) + O
\left( r^{\la_2+\la_3} \log r\right).
\end{equation}

\item[(c)]
If $p=p_c$, finally, then $\la_2=\la_3$ and
\begin{equation}\label{main3}
r^{\frac{4}{p-1}} \phi(r) = a_0 + r^{\la_3} (b_1\log r + a_1) + b_2 r^{2\la_3} (\log r)^2 + O \left( r^{2\la_3} \log r\right).
\end{equation}

\end{itemize}
\end{theorem}

Preliminary versions of these expansions appeared in \cite{Wei2, Win}.  The expansion in \cite{Win} only lists two terms, but
its proof is quite different from ours and contains some nice ideas.  The expansion in \cite{Wei2} lists three terms, but it
is not rigorously proven and not entirely correct in the critical case $p=p_c$.

The proof of Theorem \ref{main} is given in section \ref{aeai}; it heavily relies on the fact that $r^{\frac{4}{p-1}}\phi(r)$
is increasing, as first observed by the author \cite{Kar}.  Although a similar statement holds in the second-order case
\cite{Li2}, the corresponding proofs \cite{GNW, Li} do not use that statement.  Finally, section~\ref{basic} collects some
basic facts about the quartic polynomial \eqref{pc4} and the eigenvalues of the associated linearized equation; we use these
facts in the proof of our main result.

\section{Asymptotic expansion at infinity}\label{aeai}
In this section, we give the proof of Theorem \ref{main} regarding the asymptotic expansion of the positive, radial solutions
of \eqref{ee4}.  First, we use an Emden-Fowler transformation to transform \eqref{ee4} into an ODE whose linear part has
constant coefficients.  Then, we analyze this ODE using some key results of Gazzola-Grunau \cite{GG} and the author
\cite{Kar}.

\begin{lemma}\label{bm1}
Let $p>1$ and $m=\frac{4}{p-1}$.  If $\phi$ is a positive, radial solution of \eqref{ee4}, then
\begin{equation}\label{W}
W(s) = e^{m s} \phi(e^s) = r^m \phi(r), \quad\quad s= \log r= \log |x|
\end{equation}
is a solution of the ordinary differential equation
\begin{equation}
Q_4(m-\d_s) \,W(s) = W(s)^p,
\end{equation}
where $Q_4$ is the quartic polynomial defined by
\begin{equation}\label{Q}
Q_4(\al) = |x|^{\al+4} \Delta^2 |x|^{-\al} = \al(\al+2)(\al+2-n)(\al+4-n).
\end{equation}
\end{lemma}

\begin{proof}
Since $\d_r = e^{-s}\d_s$, a short computation allows us to write the radial Laplacian as
\begin{equation*}
\Delta = \d_r^2 + (n-1) r^{-1}\d_r = e^{-2s} (n-2+\d_s) \d_s.
\end{equation*}
Using the operator identity $\d_s e^{-ks} = e^{-ks}(\d_s-k)$, one can then easily check that
\begin{align*}
\Delta^2 e^{-m s} &= e^{-4s -m s} \,Q_4(m-\d_s) = e^{-m p s} \,Q_4(m-\d_s).
\end{align*}
This also implies that $Q_4(m-\d_s) \,W(s) = e^{m ps} \Delta^2 \phi (e^s) = W(s)^p$, as needed.
\end{proof}

\begin{lemma}\label{bm2}
Suppose that $n>4$ and $p>\frac{n+4}{n-4}$.  Then the positive, radial solutions of \eqref{ee4} form a one-parameter family
$\{ \phi_\al \}_{\al>0}$, where each $\phi_\al$ satisfies $\phi_\al(0)= \al$ and
\begin{equation}\label{phia}
\lim_{r\to \infty} r^m \phi_\al(r) = Q_4(m)^\frac{1}{p-1}
\end{equation}
with $m=\frac{4}{p-1}$ and $Q_4$ as in \eqref{Q}.  If we also assume that $p\geq p_c$, then
\begin{equation} \label{Y}
Y = r^m \phi_\al(r) - Q_4(m)^\frac{1}{p-1}
\end{equation}
is strictly increasing for all $r>0$, hence also negative for all $r>0$.
\end{lemma}

\begin{proof}
See Theorem 1 in \cite{GG} for the existence part, Theorem 3 in \cite{GG} for a proof of \eqref{phia} and equation (4.7) in
\cite{Kar} for the monotonicity of $Y$ in the supercritical case.
\end{proof}

To understand the behavior of the function $Y$ in \eqref{Y}, we use Lemma \ref{bm1} to get
\begin{equation} \label{g}
\Bigl[ Q_4(m-\d_s) - pQ_4(m) \Bigr] Y(s) = (Y+L)^p - L^p - pL^{p-1}Y \equiv g(Y),
\end{equation}
where $s=\log r$ and $L=Q_4(m)^\frac{1}{p-1}$.  Note that the linearized equation is given by the left hand side.  As we shall
show in Lemma \ref{fact}, the associated eigenvalues are all real in the supercritical case $p\geq p_c$ and they also satisfy
\begin{equation}\label{la}
\la_1 < \la_2 \leq \la_3 < 0 < \la_4.
\end{equation}
The presence of a positive eigenvalue is likely to complicate matters because we are seeking an expansion as $s\to\infty$.  We
thus isolate this eigenvalue and we factor \eqref{g} as
\begin{equation} \label{eq1}
(\d_s-\la_1)(\d_s-\la_2)(\d_s-\la_3)Z(s) = g(Y), \qquad Z(s) \equiv Y'(s) - \la_4 Y(s).
\end{equation}
Since $Y(s)$ is negative and increasing by Lemma \ref{bm2}, we actually have
\begin{equation}\label{neat}
Z(s) = |Y'(s)| + |\la_4 Y(s)|
\end{equation}
and we can use Gronwall-type estimates to control $Z(s)$; this is where Lemma \ref{bm2} becomes crucial, as it ensures that
$Z(s)\geq 0$. Once we have some precise estimate for $Z(s)$, we can simply integrate to get an estimate for $Y(s)$, and we can
then repeatedly use the following lemma to obtain refined expansions for both $Z(s)$ and $Y(s)$.

\begin{lemma}\label{expl}
Suppose $n\geq 13$ and $p\geq p_c$.  Let $Z(s)$ be as in \eqref{eq1}.   Given any $s_0\in\R$ then, there exist some constants
$\al_i,\b_i$ such that
\begin{equation}\label{Zex1}
Z(s) = \sum_{i=1}^3 \al_i e^{\la_i s} + \b_i \int_{s_0}^s e^{\la_i(s-\tau)} \cdot g(Y(\tau)) \,d\tau
\end{equation}
in the supercritical case $p>p_c$ and
\begin{align}\label{Zex2}
Z(s) &= \sum_{i=1}^2 \al_i e^{\la_i s} + \b_i \int_{s_0}^s e^{\la_i(s-\tau)} \cdot g(Y(\tau)) \,d\tau \notag\\
&\qquad + \al_3se^{\la_3s} + \b_3 \int_{s_0}^s (s-\tau) e^{\la_3(s-\tau)} \cdot g(Y(\tau)) \,d\tau
\end{align}
in the critical case $p=p_c$.  Moreover, each $\al_i$ depends on $s_0$ and the eigenvalues $\la_1,\la_2,\la_3$ which appear in
equation \eqref{eq1}, whereas each $\b_i$ depends solely on the eigenvalues.
\end{lemma}

\begin{proof}
We multiply the first equation in \eqref{eq1} by $e^{-\la_1s}$ and we integrate to get
\begin{align*}
(\d_s - \la_2)(\d_s - \la_3) Z(s) = A_1 e^{\la_1 s} + \int_{s_0}^s e^{\la_1(s-\tau)} \cdot g(Y(\tau)) \,d\tau
\end{align*}
for some constant $A_1$.  Repeating the same argument once again, we arrive at
\begin{align*}
(\d_s - \la_3) Z(s) = B_1 e^{\la_1 s} + B_2 e^{\la_2s} +
\int_{s_0}^s \int_{s_0}^\rho e^{\la_2(s-\rho)} e^{\la_1(\rho-\tau)} \cdot g(Y(\tau)) \,d\tau\,d\rho
\end{align*}
because $\la_1<\la_2$.  Once we now switch the order of integration, we get
\begin{align*}
(\d_s - \la_3) Z(s)
&= B_1e^{\la_1 s} + B_2 e^{\la_2s} + \int_{s_0}^s \int_\tau^s e^{\la_2(s-\rho)} e^{\la_1(\rho-\tau)} \cdot g(Y(\tau)) \,d\rho\,d\tau \\
&= B_1e^{\la_1 s} + B_2 e^{\la_2s} + \int_{s_0}^s \frac{e^{\la_1(s-\tau)} - e^{\la_2(s-\tau)}}{\la_1 - \la_2} \cdot g(Y(\tau))\,d\tau.
\end{align*}
In the supercritical case, $\la_2<\la_3$ by Lemma \ref{fact}, so we can repeat our approach once again to deduce
\eqref{Zex1}. In the critical case, $\la_2=\la_3$ so our approach leads to \eqref{Zex2}.
\end{proof}

\begin{lemma}\label{1st}
Suppose $n\geq 13$ and $p\geq p_c$.  Let $Z(s)$ be as in \eqref{eq1} and let $\de>0$.  Then
\begin{equation}\label{11}
Z(s)= O \left( e^{\la_3s+\de s} \right) \quad\text{as\, $s\to\infty$.}
\end{equation}
\end{lemma}

\begin{proof}
Suppose first that $p>p_c$.  Then $\la_1< \la_2< \la_3$ by Lemma \ref{fact} and so
\begin{equation} \label{12}
Z(s) \leq C_1e^{\la_3s} + C_2 \int_{s_0}^s e^{\la_3(s-\tau)} \cdot |g(Y(\tau))| \,d\tau
\end{equation}
by \eqref{Zex1}.  Now, Lemma \ref{bm2} and our definition \eqref{g} ensure that
\begin{equation*}
\lim_{\tau\to\infty} \frac{g(Y(\tau))}{Y(\tau)} = g'(0) = 0.
\end{equation*}
Since $Z(s) \geq |\la_4 Y(s)|$ by equation \eqref{neat}, this trivially implies
\begin{equation*}
|g(Y(\tau))| \leq \frac{\de|\la_4Y(\tau)|}{C_2} \leq \frac{\de Z(\tau)}{C_2}
\end{equation*}
for all large enough $\tau$.  Inserting this estimate in \eqref{12}, we conclude that
\begin{equation}\label{13}
e^{-\la_3s} Z(s) \leq C_1 + \de \int_{s_0}^s e^{-\la_3 \tau} Z(\tau) \,d\tau
\end{equation}
for all large enough $s_0,s$.  Using Gronwall's inequality, we thus obtain \eqref{11}.

Suppose now that $p=p_c$, in which case $\la_1< \la_2=\la_3$ by Lemma \ref{fact}.  Using the exact same approach as above with
\eqref{Zex2} instead of \eqref{Zex1}, we now get
\begin{equation*}
e^{-\la_3 s} Z(s) \leq C_1s +  \frac{\de}{2} \int_{s_0}^s e^{-\la_3\tau} Z(\tau) \,d\tau +
\frac{\de^2}{2} \int_{s_0}^s (s-\tau) e^{-\la_3\tau} Z(\tau) \,d\tau
\end{equation*}
for all large enough $s_0,s$ in analogy with \eqref{13}.  Letting $F(s)$ denote the rightmost integral, one can express the
last equation in the form
\begin{equation}\label{14}
F''(s) \leq C_1s + \frac{\de}{2} F'(s) + \frac{\de^2}{2} F(s), \qquad F(s_0)=F'(s_0)=0.
\end{equation}
We note that $F,F'$ are non-negative, while $G(s) = F'(s) + \frac{\de}{2} F(s)$ is such that
\begin{equation*}
G'(s) - \de G(s) = F''(s) - \frac{\de}{2} F'(s) - \frac{\de^2}{2} F(s) \leq C_1s.
\end{equation*}
Multiplying by $e^{-\de s}$ and integrating, we now get
\begin{equation*}
F'(s) + \frac{\de}{2} F(s) = G(s) \leq C_1e^{\de s} \int_{s_0}^s \tau e^{-\de\tau} \,d\tau \leq C_2 e^{\de s}.
\end{equation*}
Since $F,F'$ are non-negative by above, we may thus recall \eqref{14} to conclude that
\begin{equation*}
e^{-\la_3s} Z(s) = F''(s) \leq C_1s + C_3e^{\de s}.
\end{equation*}
This trivially implies the desired \eqref{11} and also completes the proof.
\end{proof}

\begin{lemma}
Under the assumptions of the previous lemma, one has
\begin{equation}\label{Z1}
Z(s)= \left\{ \begin{array}{llc}
O(se^{\la_3s}) &\text{if\, $p=p_c$} \\
O(e^{\la_3s}) &\text{if\, $p>p_c$}
\end{array} \right\}.
\end{equation}
\end{lemma}

\begin{proof}
Let us fix some $0< \de< |\la_3|/2$ and consider two cases.

\Case{1} When $p>p_c$, we use equation \eqref{Zex1} to get
\begin{equation*}
Z(s) \leq C_1 e^{\la_3 s} + C_2 \int_{s_0}^s e^{\la_3(s-\tau)} \cdot |g(Y(\tau))| \,d\tau.
\end{equation*}
According to \eqref{g} and \eqref{neat}, the rightmost factor in the integrand is bounded by
\begin{equation*}
|g(Y(\tau))| \leq CY(\tau)^2 \leq CZ(\tau)^2
\end{equation*}
for all large enough $\tau$.  Using this fact and the previous lemma, we find that
\begin{equation}\label{15}
e^{-\la_3 s} Z(s) \leq C_1 + C \int_{s_0}^s e^{(\la_3 +2\de)\tau} \,d\tau
\end{equation}
for all large enough $s_0,s$.  Since $2\de<-\la_3$, the result now follows.

\Case{2} When $p=p_c$, we use \eqref{Zex2} instead of \eqref{Zex1}.  Proceeding as above, one gets
\begin{align*}
e^{-\la_3 s} Z(s) \leq C_3s + C_3 \int_{s_0}^s e^{(\la_3 +2\de)\tau} \,d\tau + C_3
\int_{s_0}^s (s-\tau) e^{(\la_3 +2\de)\tau} \,d\tau
\end{align*}
instead of \eqref{15}.  Since $2\de<-\la_3$, the result follows as before.
\end{proof}

\begin{corollary}\label{2nd}
Under the assumptions of Lemma \ref{1st}, one also has
\begin{equation}\label{Y1}
Y(s)= \left\{ \begin{array}{llc}
O(se^{\la_3s}) &\text{if\, $p=p_c$} \\
O(e^{\la_3s}) &\text{if\, $p>p_c$}
\end{array} \right\}.
\end{equation}
\end{corollary}

\begin{proof}
Recall our definition \eqref{eq1} which reads
\begin{equation*}
\Bigl[ e^{-\la_4s} Y(s) \Bigr]' = e^{-\la_4s} Z(s)
\end{equation*}
for some $\la_4>0$.  Since $Y(s)\to 0$ as $s\to \infty$ by \eqref{phia}, we may then integrate to get
\begin{equation}\label{Yex}
Y(s) = -\int_s^\infty e^{\la_4(s-\tau)} Z(\tau) \,d\tau.
\end{equation}
Using the expansion \eqref{Z1} for $Z(\tau)$, we obtain the expansion \eqref{Y1} for $Y(s)$.
\end{proof}

\begin{rem}
In what follows, we shall frequently use the fact that
\begin{equation*}
\int_{s_0}^s O(e^{\mu\tau}) \,d\tau =
\int_{s_0}^\infty O(e^{\mu\tau}) \,d\tau - \int_s^\infty O(e^{\mu\tau}) \,d\tau = C_1 + O(e^{\mu s})
\end{equation*}
whenever $\mu<0$ as well as the analogous statement
\begin{equation*}
\int_{s_0}^s O(e^{\mu\tau}) \,d\tau = O(e^{\mu s})
\end{equation*}
whenever $\mu>0$.  Moreover, similar statements hold with $e^{\mu s}$ replaced by $se^{\mu s}$.
\end{rem}

\begin{proof_of}{Theorem \ref{main}}
Our assertions about the eigenvalues $\la_i$ and the critical values $p_i$ are basically facts about the quartic polynomial
\eqref{Q}, so we establish them separately in section~\ref{basic}, see Lemmas \ref{fact} and \ref{pis}, respectively.  In
what follows, we may thus focus solely on the asymptotic expansions stated in the theorem.

For part (a), we assume $p_k < p < p_{k+1}$.  In this case, we shall prove the expansions
\begin{equation}\label{a|1}
Y(s) = \sum_{j=1}^l a_j e^{j\la_3s} + O(e^{(l+1)\la_3s}), \qquad 0\leq l\leq k-1
\end{equation}
with the sum interpreted as zero when $l=0$, as well as the refined expansion
\begin{equation}\label{a|2}
Y(s) = \sum_{j=1}^k a_j e^{j\la_3s} + b_1 e^{\la_2s} + O(e^{(k+1)\la_3s}).
\end{equation}
Note that the former holds when $l=0$ by Corollary \ref{2nd}.  To establish them both at the same time, we will show that
\eqref{a|1} with $l< k-1$ implies \eqref{a|1} with $l+1$ and that \eqref{a|1} with $l=k-1$ implies \eqref{a|2}.

Suppose then that \eqref{a|1} holds for some $0\leq l\leq k-1$.  In view of \eqref{g}, the Taylor series expansion of $g$ near
$Y=0$ has the form
\begin{equation}\label{Tayl}
g(Y)= \sum_{j=2}^L d_jY^j + O \left( Y^{L+1} \right),
\end{equation}
where $L$ is an arbitrary positive integer.  We take $L=l+1$ and use \eqref{a|1} to get
\begin{equation*}
g(Y(s)) = \sum_{j=2}^{l+1} c_j e^{j\la_3s} + O \left( e^{(l+2)\la_3s} \right).
\end{equation*}
The corresponding expression for $Z(s)$ provided by Lemma \ref{expl} is
\begin{equation}\label{expZ}
Z(s) = \sum_{i=1}^3 \al_i e^{\la_i s} + \b_i \int_{s_0}^s e^{\la_i(s-\tau)} \cdot g(Y(\tau)) \,d\tau
\end{equation}
and we may combine the last two equations to arrive at
\begin{equation}\label{exp1}
Z(s) = \sum_{i=1}^3 \al_i e^{\la_i s} + \b_i e^{\la_is} \int_{s_0}^s
\left[ \sum_{j=2}^{l+1} c_j e^{(j\la_3-\la_i)\tau} + O
\left( e^{(l+2)\la_3\tau-\la_i\tau} \right) \right] d\tau.
\end{equation}
Since $p_k < p < p_{k+1}$ by assumption, Lemmas \ref{fact} and \ref{pis} ensure that
\begin{equation}\label{lais}
\la_1 < \la_2 + \la_3 < (k+1)\la_3 < \la_2 < k\la_3 <0.
\end{equation}
In particular, $j\la_3-\la_1$ is positive for each $j\leq l+2\leq k+1$, whereas $j\la_3-\la_2$ is positive when $j\leq k$ and
negative when $j\geq k+1$.  Thus, \eqref{exp1} leads to the expansion
\begin{equation*}
Z(s) = \sum_{j=1}^{l+1} a_j e^{j\la_3s} + O \left( e^{(l+2)\la_3s} \right),
\end{equation*}
if $l\leq k-2$, but it leads to the expansion
\begin{equation*}
Z(s) = \sum_{j=1}^k a_j e^{j\la_3s} + b_1 e^{\la_2 s} + O(e^{(k+1)\la_3s}),
\end{equation*}
if $l=k-1$.  In either case, a similar expansion is easily seen to hold for
\begin{equation*}
Y(s) = -\int_s^\infty e^{\la_4(s-\tau)} Z(\tau) \,d\tau
\end{equation*}
by \eqref{Yex}.  This shows that \eqref{a|1} with $l<k-1$ implies \eqref{a|1} with $l+1$ and that \eqref{a|1} with $l=k-1$
implies \eqref{a|2}.  In particular, \eqref{a|2} follows by induction.

We now repeat this argument to refine \eqref{a|2} even further.  As before, we insert \eqref{a|2} in \eqref{Tayl} with $L=k+1$
and we end up with
\begin{equation*}
g(Y(s)) = \sum_{j=2}^{k+1} c_j e^{j\la_3s} + O \left( e^{(\la_2+ \la_3)s} \right).
\end{equation*}
Inserting the last equation in \eqref{expZ} and recalling \eqref{lais}, we deduce that
\begin{equation*}
Z(s) = \sum_{j=1}^k a_j e^{j\la_3s} + b_1 e^{\la_2s} + a_{k+1}e^{(k+1)\la_3s} +  O \left( e^{(\la_2+ \la_3)s} \right).
\end{equation*}
By \eqref{Yex}, a similar expansion holds for $Y(s)$, so the expansion \eqref{main1} for part (a) follows.

For part (b), we assume $p=p_k$ with $k\geq 2$.  In this case, we shall similarly prove
\begin{equation}\label{b|1}
Y(s) = \sum_{j=1}^l a_j e^{j\la_3s} + O \left( e^{(l+1)\la_3s} \right), \qquad 0\leq l\leq k-2
\end{equation}
together with the refined expansion
\begin{equation}\label{b|2}
Y(s) = \sum_{j=1}^{k-1} a_j e^{j\la_3s} + O \left( se^{k\la_3s} \right).
\end{equation}
Note that the former holds when $l=0$ by Corollary \ref{2nd}.  Proceeding as before, we assume \eqref{b|1} holds for some
$0\leq l\leq k-2$ and insert \eqref{b|1} in \eqref{Tayl} with $L=l+1$.  As \eqref{b|1} is the same expansion that we had
before, we still end up with \eqref{exp1}, but we now have
\begin{equation}\label{lais2}
\la_1 < \la_2 + \la_3 = (k+1)\la_3 < \la_2 = k\la_3 <0
\end{equation}
instead of \eqref{lais}.  Assuming $l< k-2$, equation \eqref{exp1} leads to \eqref{b|1} with $l+1$ exactly as before.  If
$l=k-2$, on the other hand, then it leads to \eqref{b|2}.

We now repeat this argument to refine \eqref{b|2} even further.  Inserting \eqref{b|2} in \eqref{Tayl} with $L=k$, one finds
that
\begin{equation*}
g(Y(s)) = \sum_{j=2}^k c_j e^{j\la_3s} + O \left( se^{(k+1)\la_3s} \right).
\end{equation*}
We combine this fact with \eqref{expZ} and recall \eqref{lais2} to get
\begin{equation*}
Z(s) = \sum_{j=1}^{k-1} a_j e^{j\la_3s} + b_1 se^{k\la_3s} + a_ke^{k\la_3s} + O \left( se^{(k+1)\la_3s} \right).
\end{equation*}
Once again, this also implies the desired expansion \eqref{main2} for part (b).

For part (c), finally, we assume $p=p_1$.  In this case, \eqref{b|2} with $k=1$ is already known to hold by Corollary
\ref{2nd}. To refine this expansion, we insert it in \eqref{Tayl} with $L=1$ and proceed as before to find that
\begin{equation*}
Z(s) = b_1 se^{\la_3 s} + a_1 e^{\la_3s} + O \left( s^2e^{2\la_3s} \right).
\end{equation*}
Using this and \eqref{Tayl} with $L=2$, we may then repeat the same approach once more to end up with the expansion
\eqref{main3} which is stated in the theorem.
\end{proof_of}

\section{Useful facts}\label{basic}
In this section, we gather some basic facts related to the quartic polynomial \eqref{Q} in the supercritical case $p\geq
p_c>\frac{n+4}{n-4}$, in which case
\begin{equation} \label{las}
\la_* \equiv \frac{4}{p-1} - \frac{n-4}{2} < 0.
\end{equation}

\begin{lemma}\label{fact}
Suppose $n\geq 13$ and $p\geq p_c$.  If $m=\frac{4}{p-1}$ and $Q_4$ is the quartic \eqref{Q}, then
\begin{equation}\label{pol}
\mathscr{P}(\la) = Q_4(m-\la) -pQ_4(m)
\end{equation}
has four real roots $\la_1 < 2\la_* < \la_2 \leq \la_* \leq \la_3 < 0 < \la_4$ and those are such that
\begin{equation}\label{roots}
\la_1 + \la_ 4 = \la_2 + \la_3 = 2\la_*.
\end{equation}
Moreover, a double root arises if and only if $p=p_c$, in which case $\la_2=\la_3=\la_*$.
\end{lemma}

\begin{proof}
Noting that $Q_4$ is symmetric about $\frac{n-4}{2}$, we see that $\mathscr{P}$ is symmetric about
\begin{equation*}
\la_* \equiv m- \frac{n-4}{2} = \frac{4}{p-1} - \frac{n-4}{2} < 0.
\end{equation*}
In addition, we have
\begin{equation}\label{pol1}
\lim_{\la\to \pm\infty} \mathscr{P}(\la) = + \infty, \qquad
\mathscr{P}(2\la_*) = \mathscr{P}(0) = (1-p) \cdot Q_4(m) < 0
\end{equation}
and we also have
\begin{equation}\label{pol2}
\mathscr{P}(\la_*) = Q_4\left( \frac{n-4}{2} \right) - p\cdot Q_4\left( \frac{4}{p-1} \right) \geq 0
\end{equation}
because $p\geq p_c$.  This forces $\mathscr{P}(\la)$ to have at least one root in each of the intervals
\begin{equation*}
(-\infty,2\la_*), \qquad (2\la_*,\la_*], \qquad [\la_*,0), \qquad (0,\infty).
\end{equation*}
If $\la_*$ happens to be a root, then it must be a double root by symmetry; this is only the case when equality holds in
\eqref{pol2}, namely when $p=p_c$.  As for our assertion \eqref{roots}, this also follows by symmetry because $\la$ is a root
of $\mathscr{P}$ if and only if $2\la_*-\la$ is.
\end{proof}

\begin{lemma}\label{p1}
Suppose $n\geq 13$ and $p\geq p_c$.  Let $\la_2\leq \la_3<0$ be as in Lemma \ref{fact} and let $Q_4$ be the quartic polynomial
\eqref{Q}.  Given any integer $k\geq 1$, we then have
\begin{equation}\label{rel}
\la_2 > k\la_3 \quad\iff\quad \mathscr{R}_k(p) < 0,
\end{equation}
where $\mathscr{R}_k(p)$ is the polynomial defined by
\begin{equation}\label{Rk}
\mathscr{R}_k(p) = (p-1)^4 \left[ Q_4 \left( \frac{k-1}{k+1} \cdot \frac{4}{p-1} + \frac{n-4}{k+1} \right)
- pQ_4 \left( \frac{4}{p-1} \right) \right].
\end{equation}
\end{lemma}

\begin{proof}
First of all, we use our identity \eqref{roots} to find that
\begin{equation*}
\la_2 > k\la_3 \quad\iff\quad 2\la_* = \la_2 + \la_3 > (k+1)\la_3,
\end{equation*}
where $\la_*<0$ is defined by \eqref{las}.  According to Lemma \ref{fact}, $\la_3$ is the unique root of
\begin{equation*}
\mathscr{P}(\la) = Q_4 \left( \frac{4}{p-1} -\la \right) -p Q_4 \left( \frac{4}{p-1} \right)
\end{equation*}
that lies in the interval $[\la_*,0)$.  Since $k\geq 1$ by assumption, $\frac{2\la_*}{k+1}$ also lies in that interval, while
equations \eqref{pol1} and \eqref{pol2} give $\mathscr{P}(\la_*)\geq 0> \mathscr{P}(0)$.  In particular, we have
\begin{equation*}
\la_2 > k\la_3 \quad\iff\quad \frac{2\la_*}{k+1} > \la_3 \quad\iff\quad \mathscr{P} \left( \frac{2\la_*}{k+1} \right) < 0,
\end{equation*}
where $\la_*= \frac{4}{p-1} - \frac{n-4}{2}$ by definition \eqref{las}, so the desired condition \eqref{rel} follows.
\end{proof}

\begin{lemma}\label{p2}
Suppose $n\geq 13$ and $p\geq p_c$.  Let $k>1$ be an integer, let $\mathscr{R}_k(p)$ denote the polynomial \eqref{Rk} of the
previous lemma, and let $L$ denote the limit
\begin{equation}\label{lim}
L= \lim_{p\to \pm\infty} \frac{\mathscr{R}_k(p)}{p^4} = Q_4 \left( \frac{n-4}{k+1} \right) - 8(n-2)(n-4).
\end{equation}

\begin{itemize}
\item[(a)]
If $L\leq 0$, then $\mathscr{R}_k(p)$ is negative on $[p_c,\infty)$.
\item[(b)]
If $L>0$, then $\mathscr{R}_k(p)$ has a unique root $p_k>p_c$ and it is negative on $[p_c,p_k)$ but positive on
$(p_k,\infty)$.
\end{itemize}
\end{lemma}

\begin{proof}
First, we show that $\mathscr{R}_k(p)$ has two roots in $(1,p_c)$ by showing that
\begin{equation}\label{R1}
\mathscr{R}_k(1) < 0, \qquad \mathscr{R}_k \left( \frac{n}{n-4} \right) > 0, \qquad \mathscr{R}_k(p_c) < 0.
\end{equation}
Using our definition \eqref{Q}, one can easily check that
\begin{equation}\label{ap}
\lim_{p\to 1} \,(p-1)^4 \cdot Q_4 \left( \frac{a}{p-1} \right) = a^4
\end{equation}
for any $a\in\R$ whatsoever; in view of our definition \eqref{Rk}, this also implies
\begin{equation*}
\mathscr{R}_k(1) = 4^4 \left( \frac{k-1}{k+1} \right)^4 - 4^4 < 0.
\end{equation*}
Next, we note that $Q_4$ is positive on $(0,n-4)$ with $Q_4(n-4)=0$, hence
\begin{align*}
\mathscr{R}_k \left( \frac{n}{n-4} \right) &= \left( \frac{4}{n-4} \right)^4 \cdot Q_4 \left( \frac{k(n-4)}{k+1} \right) >0.
\end{align*}
To show that $\mathscr{R}_k(p_c) < 0$, finally, we combine Lemma \ref{fact} with \eqref{rel}.  When $p=p_c$, the lemma gives
$\la_2=
\la_3<0$, hence $\la_2>k\la_3$ and so $\mathscr{R}_k(p_c) < 0$ by \eqref{rel}.

This completes the proof of \eqref{R1}, which implies that $\mathscr{R}_k(p)$ has two roots in $(1,p_c)$.  To find the
location of the remaining roots, we shall now have to distinguish two cases.

\Case{1} When $n\leq 2(k+1)$, there are two additional roots in $[-1,1)$ because
\begin{equation}\label{R2}
\mathscr{R}_k(-1)\leq 0, \qquad \mathscr{R}_k(-1/3) > 0, \qquad \mathscr{R}_k(1)<0.
\end{equation}
Assuming this statement for the moment, $\mathscr{R}_k$ has no roots in $[p_c,\infty)$, so it must be negative there by
\eqref{R1}, and thus the limit \eqref{lim} is non-positive; in fact, the limit has to be negative, as it can only be zero when
$\mathscr{R}_k$ is a cubic instead of a quartic.

To finish the proof for this case, we now establish \eqref{R2}.  First of all, we have
\begin{equation*}
\mathscr{R}_k(-1) = 2^4 \cdot Q_4 \left( \frac{n}{k+1} - 2 \right) \leq 0
\end{equation*}
because $Q_4$ is non-positive on $(-2,0]$.  Since $\mathscr{R}_k(1)<0$ by \eqref{R1}, it remains to show
\begin{align*}
(3/4)^4 \cdot \mathscr{R}_k(-1/3) &= Q_4\left( \frac{n+2}{k+1} - 3 \right) + \frac{Q_4(-3)}{3} \\
&= Q_4\left( \frac{n}{k+1} - 2 - \frac{k-1}{k+1} \right) + n^2-1
\end{align*}
is positive.  In particular, it suffices to show that
\begin{equation}
Q_4(x) + n^2-1 > 0 \quad \text{whenever\, $-3< x<0$.}
\end{equation}
Since $Q_4$ is positive on $(-3,-2)$, we may assume $-2\leq x<0$.  Then
\begin{equation*}
x(x+2) \geq -1 \quad\Longrightarrow\quad Q_4(x)\geq -(x+2-n)(x+4-n)
\end{equation*}
and the rightmost quadratic is increasing on $(-\infty,0)$, so we easily get
\begin{equation*}
Q_4(x)\geq n(2-n) > 1-n^2.
\end{equation*}

\Case{2} When $n>2(k+1)$, there is one additional root in $(-1,1)$ because
\begin{equation}\label{R3}
\mathscr{R}_k(-1)> 0, \qquad \mathscr{R}_k(1) < 0.
\end{equation}
This follows easily by \eqref{R1} and the fact that $Q_4$ is positive on $(0,n-4)$, which gives
\begin{equation*}
\mathscr{R}_k(-1) = 2^4 \cdot Q_4 \left( \frac{n}{k+1} - 2 \right) > 0
\end{equation*}
for this case.  In view of \eqref{R1}, we now know that $\mathscr{R}_k$ has three roots in $(-1,p_c)$, being positive at the
left endpoint and negative at the right endpoint.  If the limit \eqref{lim} is positive, then $\mathscr{R}_k$ is positive as
$p\to\infty$, so the fourth root lies in $(p_c,\infty)$; if the limit is zero, then $\mathscr{R}_k$ is a cubic, so it has no
other roots; and if the limit is negative, then $\mathscr{R}_k$ is negative as $p\to -\infty$, so the fourth root lies in
$(-\infty,-1)$.  The result now follows.
\end{proof}

\begin{lemma}\label{pis}
Suppose $n\geq 13$ and $p\geq p_c$.  Let $\la_2\leq \la_3<0$ be as in Lemma \ref{fact}.  Then there exists a finite sequence
$p_c= p_1 < p_2 < \cdots < p_N$ such that $\la_2 > (N+1)\la_3$ and
\begin{equation*}
\la_2 \leq k\la_3 \quad\iff\quad p\geq p_k
\end{equation*}
for each $1\leq k\leq N$.  Moreover, if $\lfloor x\rfloor$ denotes the greatest integer in $x$, then the length of this finite
sequence is
\begin{equation*}
N= \left\{ \begin{array}{llc}
\lfloor \frac{n-10}{2}\rfloor &&\text{when\, $13\leq n\leq 19$}\\ \\
\lfloor \frac{n-9}{2}\rfloor &&\text{when\, $n\geq 20$}
\end{array}\right\}.
\end{equation*}
\end{lemma}

\begin{proof}
According to the previous two lemmas, it suffices to show that the limit \eqref{lim} is positive when $k\leq N$ but negative
when $k\geq N+1$. Let us thus focus on the quartic
\begin{align}\label{F}
\mathscr{F}(k) \equiv \frac{2(k+1)^4}{n-4} \cdot \left[ Q_4 \left( \frac{n-4}{k+1} \right) - 8(n-2)(n-4) \right]
\end{align}
which is merely the limit \eqref{lim} times a positive factor.

To show that $\mathscr{F}(k)$ has three roots in the interval $(1-n/2,1)$, we show that
\begin{equation}\label{F1}
\mathscr{F}(1-n/2)<0, \qquad \mathscr{F}(-1)>0, \qquad \mathscr{F}(0) < 0, \qquad \mathscr{F}(1)>0.
\end{equation}
First of all, we have $Q_4(-2)=Q_4(n-4)=0$, so we easily get
\begin{equation*}
k= 0,1-n/2 \quad\Longrightarrow\quad \mathscr{F}(k) = -\frac{2(k+1)^4}{n-4} \cdot 8(n-2)(n-4) <0.
\end{equation*}
Using our definitions \eqref{Q} and \eqref{F}, we can also verify that
\begin{align*}
\mathscr{F}(-1) = \lim_{k\to -1} \frac{2(k+1)^4}{n-4} \cdot Q_4\left( \frac{n-4}{k+1} \right) = 2(n-4)^3 >0,
\end{align*}
while the fact that $\mathscr{F}(1)>0$ follows by the Taylor expansion
\begin{align*}
\mathscr{F}(1) &= 2n^3 - 8n^2 - 256n+ 512 \\
&= 2(n-13)^3 +70(n-13)^2 + 550(n-13) + 226.
\end{align*}
This proves \eqref{F1}, which implies that $\mathscr{F}$ has three roots in $(1-n/2,1)$.

To see that the fourth root lies in $(n/2-5, n/2-4)$, we now note that
\begin{align*}
\mathscr{F}(n/2-5) &= 2n^4 - 60n^3 + 608n^2 - 2336n + 2432 \\
&= 2(n-13)^4 + 44(n-13)^3 + 296(n-13)^2 + 628(n-13)+ 118
\end{align*}
is positive for each $n\geq 13$, whereas
\begin{align*}
\mathscr{F}(n/2-4) &= -n^4 + 18n^3 - 124n^2 + 416n - 608 \\
&= -(n-13)^4 - 34(n-13)^3 - 436(n-13)^2 - 2470(n-13) - 5171
\end{align*}
is negative.  In particular, the fourth root of $\mathscr{F}$ must lie in $(n/2-5,n/2-4)$.

\Case{1} If $1\leq k\leq n/2-5$, then $\mathscr{F}(k)$ is positive by above.

\Case{2} If $k\geq n/2-4$, then $\mathscr{F}(k)$ is negative by above.

\Case{3} If $k= (n-9)/2$ and $n$ is odd, finally, then we can readily check that
\begin{equation*}
\mathscr{F}(k) = \frac{n-1}{2} \cdot (n^3 -33n^2 +312n -892)
\end{equation*}
is positive if and only if $n\geq 20$.  In any case then, the result follows easily.
\end{proof}

\bibliography{asymptotic}
\bibliographystyle{siam}
\end{document}